\documentclass[12pt]{amsart}
\usepackage{amssymb, amsmath}
\usepackage{fourier}
\usepackage{xcolor}

\theoremstyle{plain}

\newtheorem{theorem}{Theorem}

\newtheorem{proposition}[theorem]{Proposition}

\theoremstyle{remark}

\newcommand\C{\mathbb C}

\newcommand\N{\mathbb N}

\newcommand\hP{\mathbb P}
\renewcommand\l{\lambda}

\newcommand\tr{\operatorname{Tr}}

\begin{document}

\title[]{Maps on positive definite matrices preserving Bregman and Jensen divergences}

\author{Lajos Moln\'ar}
\address{Department of Analysis, Bolyai Institute\\
University of Szeged\\
H-6720 Szeged, Aradi v\'ertan\'uk tere 1.,
Hungary and
MTA-DE ``Lend\" ulet'' Functional Analysis Research Group, Institute of Mathematics\\
         University of Debrecen\\
         H-4010 Debrecen, P.O. Box 12, Hungary}
\email{molnarl@math.u-szeged.hu}
\urladdr{http://www.math.unideb.hu/\~{}molnarl/}
        
\author{J\'ozsef Pitrik}
\address{Institute of Mathematics\\
Budapest University of Technology and Economics\\
H-1521 Budapest, Hungary}
\email{pitrik@math.bme.hu}
\urladdr{http://www.math.bme.hu/\~{}pitrik}

\author{D\'aniel Virosztek}
\address{Institute of Mathematics\\
Budapest University of Technology and Economics\\
H-1521 Budapest, Hungary}
\email{virosz@math.bme.hu}
\urladdr{http://www.math.bme.hu/\~{}virosz}

\thanks{The first and third authors were supported by the "Lend\" ulet" Program (LP2012-46/2012) of the Hungarian Academy of Sciences. The first author was also supported by the Hungarian Scientific Research Fund (OTKA) Reg. No. K115383 and the third author was partially supported by the Hungarian Scientific Research Fund (OTKA) Reg. No.  K104206.}


\begin{abstract}
In this paper we determine those bijective maps of the set of all positive definite $n\times n$ complex matrices which preserve a given Bregman divergence corresponding to a differentiable convex function that satisfies certain conditions. We cover the cases of the most important
Bregman divergences and present the precise structure of the mentioned transformations. Similar results concerning Jensen divergences and their preservers are also given.
\end{abstract}
\maketitle
\section{Introduction}

In a series of papers \cite{ML14a, ML15a,ML15b,ML13j} the first author and his coauthors described the structures of surjective maps of the positive definite cones in matrix algebras, or in operators algebras which can be considered generalized isometries meaning that they are transformations which preserve "distances" with respect to given so-called generalized distance measures.
This latter notion stands for any function $d:\mathcal X\times \mathcal X\to [0,\infty )$ on any set $\mathcal X$ with the mere property that for $x,y\in \mathcal X$ we have $d(x,y)=0$ if and only if $x=y$. We recall that in several areas of mathematics not only metrics are used to measure nearness of points but also more general functions of this latter kind. 

In \cite{ML15b,ML13j} the considered generalized distance measures are of the form $d=d_{N,g}$, where $N(.)$ is a unitarily invariant norm on the underlying matrix algebra or operator algebra, $g:(0,\infty)\to \mathbb C$ is a continuous function with the properties  
\begin{itemize}
\item[(a1)]
$g(y)=0$ if and only if $y=1$; 
\item[(a2)]
there exists a constant $K>1$ such that
$\left|g\left(y^2\right)\right|\geq K\left|g(y)\right|$, $y>0$,
\end{itemize}
and the generalized distance measure 
$d_{N,g}$ is defined by
\begin{equation}\label{E:M4}
d_{N,g}(X,Y)=N\left(g\left(Y^{-1/2}XY^{-1/2}\right)\right)
\end{equation}
for all positive invertible elements $X,Y$ of the underlying algebra.
In the mentioned papers one can see several important examples of that sort of generalized distance measures, many of them having backgrounds in the differential geometry of positive definite matrices or operators.
The basic tools in describing the structure of the corresponding generalized isometries have been so-called generalized Mazur-Ulam type theorems and descriptions of certain algebraic isomorphisms (Jordan triple isomorphisms) of the positive definite cones in question.

In the present paper we determine the structures of generalized isometries with respect to other important types of generalized distance measures. Namely, here we consider Bregman divergences and Jensen divergences. These types of divergences have wide ranging applications in several areas of mathematics. For example, in the recent volume \cite{NB} on matrix information geometry 3 chapters are devoted to the study of Bregman divergences. One feature of Jensen divergences which justifies their importance is that Bregman divergences can be considered as asymptotic Jensen divergences (see Section 6.2 in \cite{NB}). We further mention that the famous Stein's loss and Umegaki's relative entropy are among the most important Bregman divergences. Our basic tool in this paper to determine the corresponding preserver transformations is, just as above, also algebraic in nature but rather different from what we have mentioned in the previous paragraph. Namely, here we use order isomorphisms.

Before presenting the results we fix the notation and terminology.
In what follows $\mathbb M_n$ denotes the algebra of all $n\times n$ complex matrices and $\hP_n$ stands for the positive definite cone in $\mathbb M_n$, i.e., the set of all positive definite matrices in $\mathbb M_n$. Whenever convenient, in the paper we use the equivalence of the language of matrices and that of linear operators on $\mathbb C^n$.

We next define the two basic concepts what we consider in this paper, Bregman divergence and Jensen divergence. Both concepts are connected to convex real functions.
Let $f$ be a convex function on the interval $(0,\infty)$. It is known that $f$ is necessarily continuous and the set of points where $f$ is differentiable has at most countable complement. It is a remarkable fact that if $f$ is everywhere differentiable then its derivative $f'$ is automatically continuous (see e.g. \cite[Corollary 25.5.1]{R}).

For a differentiable convex function $f$ on $(0, \infty)$, the Bregman $f$-divergence on $\hP_n$ is
\[
H_f(X,Y)=\tr \left(f(X)-f(Y)-f'(Y)(X-Y)\right), \quad X,Y\in \hP_n,
\]
see e.g. formula (5) in \cite{pv15}. If $\lim_{x \to 0+} f(x)$ and $\lim_{x \to 0+} f'(x)$ exist, then $f,f'$ have continuous extensions onto $[0,\infty )$ and the Bregman $f$-divergence is well-defined and finite for any pair of positive semidefinite matrices, too.

For a convex function $f$ on $(0, \infty)$ and for given $\lambda \in (0,1)$, the Jensen $\lambda-f$-divergence on $\hP_n$ is defined by
\[
J_{f, \lambda}(X,Y)=\tr \left(\lambda f(X)+(1-\lambda)f(Y)-f\left(\lambda X +(1-\lambda) Y\right)\right), \quad X,Y\in \hP_n.
\]
If $\lim_{x \to 0+} f(x)$ exists, then the Jensen $\lambda-f$-divergence is also well-defined and finite for any pair of positive semidefinite matrices.

It is well-known that $H_f$ and $J_{f, \lambda}$ are always nonnegative and they are generalized distance measures (in the sense we use in this paper) if and only if $f$ is strictly convex.

Our main aim is to describe the "symmetries" of the positive definite cone $\hP_n$ that preserve above type of divergences. This means that we are looking for the structure of all bijective maps $\phi:\hP_n \to \hP_n$ for which 
\[
H_f\left(\phi(X),\phi(Y)\right)=H_f(X,Y), \quad X,Y\in \hP_n
\]
or
\[
J_{\l, f}\left(\phi(X),\phi(Y)\right)=J_{\l, f}(X,Y), \quad X,Y\in \hP_n
\]
holds.
\section{The results}

We now turn to our results. It is useful to examine first the question that how different the present problem is from the ones we have considered in the papers \cite{ML14a,ML15a,ML15b,ML13j}. To see clearly the differences we determine below which Bregman divergences, respectively which Jensen divergences are of the form \eqref{E:M4}. 
Let us begin with the case of Bregman divergences.

One of the most important Bregman divergence is the one usually denoted by $l$ which corresponds to the strictly convex function $f(x)=-\log x$, $x>0$ and is called Stein's loss. Apparently, we have
\[
l(X,Y)=-\tr \left(\log X -\log Y -Y^{-1}(X-Y)\right)=
\tr XY^{-1}-\log \det XY^{-1}-n
\]
for all $X,Y\in \hP_n$, where we have used the identity $\tr \circ \log=\log \circ \det$ on $\hP_n$.
On the other hand, the continuous function $g(y)=y-\log y -1, y>0$ is nonnegative, satisfies the conditions (a1), (a2) (with constant $K=2$), and with the trace norm $\|.\|_1$ we easily get
\[ 
d_{\|.\|_1, g}(X,Y)=\left\| g\left(Y^{-1/2} XY^{-1/2}\right)\right\|_1=l(X,Y), \quad X,Y\in \hP_n.
\]

In the next proposition we prove that Stein's loss is essentially the only Bregman divergence on $\hP_n$ which is a generalized distance measure of the form \eqref{E:M4}. Observe that any distance measure of the form \eqref{E:M4} is invariant under multiplication of the variables $X,Y$ by the same positive scalar $t$.

\begin{proposition}\label{P:1}
Let $f$ be a differentiable convex function on $(0,\infty)$. Assume that the Bregman $f$-divergence $H_f$ on $\hP_n$ is homogeneous of degree 0, i.e., it satisfies
\begin{equation}\label{E:M5}
H_f(tX,tY)=H_f(X,Y)\quad X,Y\in \hP_n, t>0.
\end{equation}
Then we have $f(x)=a\log x +bx +c$, $x>0$ with some real scalars $a,b,c$ and $a\leq 0$.
\end{proposition}

\begin{proof}
We plug scalar multiples of the identity $X=xI, Y=yI$, $x,y>0$ into the equality \eqref{E:M5} and obtain
\begin{equation}\label{E:M1}
f(tx)-f(ty)-f'(ty)t(x-y)=f(x)-f(y)-f'(y)(x-y), \quad t>0.
\end{equation}
Choosing $t=1/y$ and reordering this equality we have
\[
f'(y)=(1/(x-y))(f(x)-f(y)-f(x/y)+f(1)+f'(1)((x/y)-1)))
\]
implying that $f$ is twice continuously differentiable.
Differentiating \eqref{E:M1} with respect to $x$ twice we obtain $f''(x)=f''(tx)t^2$, $x,t>0$. In particular, it follows that
$f''(t)$ is a constant multiple of $t^{-2}$, $t>0$. We infer that $f(t)=a \log t+bt +c$, $t>0$ holds with some real constants $a,b,c$. By the convexity of $f$ we have $a\leq 0$.
\end{proof}

Concerning the appearance of the function $x\mapsto bx+c$ in the above proposition we note that adding an affine function to a given convex function $f$ does not change the corresponding Bregman divergence (the same holds for Jensen divergence, too), hence that part simply does not count.

We next examine the case of Jensen divergences.
Again, let $f(x)=-\log x$, $x>0$ and pick $\l\in (0,1)$. It is easy to see that the corresponding Jensen divergence is
\[
J_{-\log, \l}(X,Y)=\log \det \left(\l X+(1-\l) Y\right)-\log \det X^\l Y^{1-\l}, \quad X,Y\in \hP_n
\]
which can also be written as 
\[
\begin{gathered}
J_{-\log, \l}(X,Y)=
\log \det \left(\l Y^{-1/2} XY^{-1/2}+(1-\l) I\right)-\log \det \left(Y^{-1/2}X Y^{-1/2}\right)^\l\\=
\tr \left(\log \left(\left(\l Y^{-1/2} XY^{-1/2}+(1-\l) I\right)\left(Y^{-1/2}X Y^{-1/2}\right)^{-\l}\right)\right), \quad X,Y\in \hP_n.
\end{gathered}
\]
These divergences are scalar multiples of the Chebbi-Moakher log-determinant $\alpha$-divergences, see \cite{ChM,ML15b}. As mentioned in \cite{ML15b} (see pages 146-147) the continuous
function $g_{\lambda}(y)= \log \left((\l y+(1-\l) )/y^\l \right)$, $y>0$ is nonnegative, satisfies (a1), (a2) (with constant $K=2$) and, moreover, we have
\[
J_{-\log, \l}(X,Y)=\left\| g\left(Y^{-1/2} XY^{-1/2}\right)\right\|_1, \quad X,Y\in \hP_n.
\]
This means that the Jensen divergences $J_{-\log, \l}$ are also of the form \eqref{E:M4}.

Let us insert the remark here that in the particular case $\l=1/2$ the above Jensen divergence was considered and called $S$-divergence in the paper \cite{Sra} of Sra. He proved the interesting fact there that the square root of this divergence is a true metric and proposed to use it as a convenient substitute for the geodesic distance $d_{\|.\|_2,\log}$ ($\|.\|_2$ standing for the Hilbert-Schmidt or Frobenius norm) originating from the natural Riemann geometric structure on $\hP_n$.

Continuing the discussion, above we have seen that the divergences $J_{-\log, \l}$ are of the form \eqref{E:M4}. In what follows we show that they are essentially the unique Jensen divergences with this property.

\begin{proposition}\label{P:2}
Let $f$ be a convex function on $(0,\infty)$ and pick a number $\l \in (0,1)$.
Assume that the corresponding Jensen $\lambda-f$ divergence is homogeneous of order 0, i.e., it satisfies
\begin{equation}\label{E:M6}
J_{f,\l}(tX,tY)=J_{f,\l}(X,Y)\quad X,Y\in \hP_n, t>0.
\end{equation}
Then we have $f(x)=a\log x +bx +c$, $x>0$ with some real scalars $a,b,c$ and $a\leq 0$.
\end{proposition}

\begin{proof}
Just as above, plugging scalar multiples of the identity $X=xI, Y=yI$, $x,y>0$ into the equality \eqref{E:M6} we have
\begin{equation}\label{E:M2}
\begin{gathered}
\l f(tx)+(1-\l)f(ty)-f(\l tx+(1-\l)ty) \\= \l f(x)+(1-\l)f(y)-f(\l x+(1-\l)y), \quad t>0.
\end{gathered}
\end{equation}
We assert that $f$ is differentiable. In fact, choosing $x=1$ we have
\[
\begin{gathered}
f(t)=(1/\l)( \l f(1) +(1-\l)f(y)- f(\l+(1-\l)y)\\ - (1-\l)f(ty)+f(\l t+(1-\l)ty)).
\end{gathered}
\]
The result \cite[11.3. Theorem]{Jarai} of J\'arai on the regularity of solutions of functional equations applies and tells us that $f$ is necessarily continuously differentiable.
Differentiating \eqref{E:M2} with respect to $t$ we have
\[
0=\l f'(tx)x+(1-\l)f'(ty)y-f'(\l tx+(1-\l)ty)(\l x+(1-\l)y)
\]
implying the equality
\[
\l f'(tx)x+(1-\l)f'(ty)y=f'(\l tx+(1-\l)ty)(\l x+(1-\l)y)
\]
for all positive $t,x,y$.
Choosing $t=1$ we infer
\[
\l f'(x)x+(1-\l)f'(y)y=f'(\l x+(1-\l)y)(\l x+(1-\l)y).
\]
This means that the function $h(x)=f'(x)x$, $x>0$ is affine and hence it is of the form $h(x)=bx+a$, $x>0$. It follows that $f(x)=a\log x+bx+c$, $x>0$ with some real scalars $a,b,c$. Again, the convexity of $f$ implies $a\leq 0$.
\end{proof} 

The above results can be considered as characterizations (hopefully a bit interesting) of the Stein's loss and the Chebbi-Moakher log-determinant $\alpha$-divergences.

We now turn to the descriptions of the corresponding generalized isometries, i.e., transformations preserving the considered divergences. We begin with a general result relating to Bregman divergences.

We first mention the easy fact that both kinds of divergences are clearly invariant under unitary and antiunitary congruence transformations.  These are maps of the form $A\mapsto UAU^*$, where $U$ is a unitary or an antiunitary operator on $\mathbb C^n$. In fact, this follows from the following. For any continuous function $f$ on $(0,\infty)$ and for any unitary or antiunitary operator $U$ on $\mathbb C^n$ we have that $f\left(UAU^*\right)=Uf(A)U^*$ holds for every positive definite $A$. This is the consequence of the fact that on the spectrum of $A$ the function $f$ coincides with a polynomial $p$ and hence we have $f\left(UAU^*\right)=p\left(UAU^*\right)=Up(A)U^*=Uf(A)U^*$. Our theorems below show that in many cases only the unitary and antiunitary congruence transformations preserve the Bregman divergence.

\begin{theorem} \label{breg}
Let $f$ be a differentiable convex function on $(0,\infty )$ such that $f'$ is bounded from below and unbounded from above. Let $\phi: \hP_n \rightarrow \hP_n$ be a bijective map which satisfies
\[
H_f\left(\phi(A), \phi(B)\right)=H_f\left(A, B\right), \quad A,B \in \hP_n.
\]
Then there exists a unitary or antiunitary operator $U: \C^n \rightarrow \C^n$ such that $\phi$ is of the form
\[
 \phi(A)=UAU^{*}, \quad  A \in \hP_n.
\]
\end{theorem}

\begin{proof}
First we recall that since $f$ is convex and everywhere differentiable, $f'$ is continuous.
By the convexity of $f,$ its derivative $f'$ is monotonically increasing. The assumption that $f'$ is bounded from below implies that $\lim_{x \to 0+} f'(x)$ exists and is finite. Hence $f'$ can be continuously extended onto $[0,\infty)$. The same holds for $f$, too. Indeed, $f(x)=f(1)+ \int_{1}^{x} f'(t) \mathrm{d}t$, $x>0$ and $\int_{1}^{x} f'(t) \mathrm{d}t$ is convergent as $x$ tends to $0$ by the boundedness of $f'$ on $(0,1]$.

In the rest of the proof we shall need the following characterization of the usual order $\leq$ on $\hP_n$.

\textbf{Claim A.}
Let $B, C \in \hP_n.$ The set
\begin{equation}\label{E:M7}
 \left\{ H_f(B,A)-H_f(C,A) | A \in \hP_n \right\}
\end{equation}
is bounded from below if and only if $B \leq C.$

To prove the claim we first compute
\begin{equation}\label{bregform}
\begin{gathered}
H_f(B,A)-H_f(C,A)\\
 =\tr f(B)-\tr f(A) -\tr f'(A)(B-A)\\-\tr f(C)+\tr f(A)+\tr f'(A)(C-A)\\
=\tr f(B) - \tr f(C) + \tr f'(A)(C-B).
\end{gathered}
\end{equation}
Let $k$ denote a lower bound of $f'.$ Assume $B \leq C$. Then by the inequality 
\[
 f'(A)\geq k I, \quad A \in \hP_n,
\]
we have that
\[
 \tr f'(A)(C-B) \geq \tr k I (C-B)= k \tr (C-B)
\]
holds for every $A \in \hP_n$ which shows that the set \eqref{E:M7} is bounded from below.
Conversely, if $B \not\leq C,$ then there exists a unit vector $x \in \C^n$ such that $\left< Bx,x\right> > \left< Cx,x \right>.$ Let $P_x$ denote the orthogonal projection onto the one-dimensional subspace generated by $x.$ For any $t > 0$ we have
\begin{equation} \label{szam}
\tr f'\left(t P_x +\left(I-P_x\right)\right)(C-B)=f'(t)\left<(C-B)x,x\right> + f'(1) \tr \left(I-P_x\right)(C-B).
\end{equation}
Since $\left<(C-B)x,x\right> < 0$ and $\left\{ f'(t) | t> 0\right\}$ is unbounded from above, hence the first term on the right hand side of \eqref{szam} is unbounded from below. By (\ref{bregform}), it follows that that the set
\[
 \left\{ H_f(B, A)-H_f(C, A) | A\in \hP_n \right \}
\]
is unbounded from below. This proves our claim.

Since the bijective map $\phi:\hP_n\to \hP_n$ preserves the Bregman $f$-divergence, using the above characterization of the order we obtain that $\phi$ is an order automorphism, i.e., 
for any $B,C\in \hP_n$ we have $B \leq C$ if and only if  $\phi(B) \leq \phi(C).$

By the result \cite[Theorem 1]{m11} of the first author, $\phi$, just as any order automorphism of $\hP_n$, is of the form
\[
 \phi(A)= T A T^{*}, \quad A \in \hP_n,
\]
where $T$ is an invertible linear or conjugate-linear operator on $\C^n.$ We may suppose that $T$ is linear, since if we are done with this, the case of conjugate-linear $T$ is not difficult to handle. 

We show that $T$ is unitary. Assume on the contrary that $T$ is not unitary. Then $T^* T \neq I$. Consider the polar decomposition $T=UP$ of $T$ where $P=\sqrt{T^* T}$ is positive definite and $U=T P^{-1}$ is unitary. Since the Bregman $f$-divergence is invariant under unitary congruences, hence we have
\begin{equation} \label{porz}
\begin{gathered}
 H_f(A,B)=H_f\left(\phi(A),\phi(B)\right)=H_f\left(TA T^*, TBT^*\right)\\
=H_f\left(UPAPU^*, UPBPU^*\right)=H_f\left(PAP,PBP\right) 
\end{gathered}
\end{equation}
for all $A,B \in \hP_n.$ Since $T$ is not unitary, hence $P \neq I,$ and thus $P$ has an eigenvalue different from 1. Without serious loss of generality we may assume that $P$ has an eigenvalue greater than 1 for the following reason. The map $A \mapsto PAP$ is a bijection of $\hP_n$ which preserves the Bregman $f$-divergence. Therefore, the inverse transformation  $A \mapsto P^{-1}A P^{-1}$ preserves the Bregman $f$-divergence as well. If $P$ does not have any eigenvalue greater than 1, $P^{-1}$ must have one.

Suppose that $Pv = \lambda v$ for some $\lambda > 1$ and unit vector $v \in \C^n.$
Let $Q_v$ be the orthogonal projection onto the one-dimensional subspace generated by $v.$ By (\ref{porz}), the transformation $A \mapsto PAP$ preserves the Bregman $f$-divergence and so does any of its powers $A \mapsto P^n AP^n$.  Hence
\[
H_f\left( \lambda^2 Q_v , Q_v \right)= H_f\left(P^n \lambda^2 Q_v P^n, P^n Q_v P^n \right)
\]
\begin{equation} \label{phat}
=H_f\left(\lambda^{2(n+1)} Q_v, \lambda^{2n} Q_v \right)
\end{equation}
holds for any $n \in \N$.

We now consider
the symmetrized Bregman $f$-divergence $H_f(A,B)+H_f(B,A)$ which can be written in the following convenient form
\[
H_f(A,B)+H_f(B,A)=\tr f(A)-\tr f(B) -\tr f'(B)(A-B)
\]
\[
+\tr f(B)-\tr f(A) - \tr f'(A)(B-A)=\tr \left(f'(A)-f'(B) \right)(A-B).
\]
By (\ref{phat}) we have
\[
H_f\left( \lambda^2 Q_v , Q_v \right)+H_f\left( Q_v , \lambda^2 Q_v \right)
\]
\[
=H_f\left(\lambda^{2(n+1)} Q_v, \lambda^{2n} Q_v \right)+H_f\left(\lambda^{2n} Q_v, \lambda^{2(n+1)} Q_v \right)
\]
\[
= \tr \left(f'(\lambda^{2(n+1)}Q_v)-f'(\lambda^{2n}Q_v) \right)\left(\lambda^{2(n+1)}Q_v-\lambda^{2n}Q_v\right)
\]
\[
=\left(f'\left(\lambda^{2(n+1)}\right)-f'\left(\lambda^{2n}\right)\right)\lambda^{2n}\left(\lambda^2-1\right)
\]
for any $n \in \N.$
This means that $\left(f'(\lambda^{2(n+1)})-f'\left(\lambda^{2n}\right)\right)\lambda^{2n}$ is independent of $n,$ that is,
\[
f'\left(\lambda^{2(n+1)}\right)-f'\left(\lambda^{2n}\right)=\frac{c}{\left(\lambda^2\right)^n} 
\]
holds for some constant $c.$
Therefore,
\[
\lim_{x \to \infty} f'(x)=\lim_{n \to \infty} f'\left(\lambda^{2(n+1)}\right)= \lim_{n \to \infty} \left(f'(1) +\sum_{k=0}^n \left( f'\left(\lambda^{2(k+1)}\right)-f'\left(\lambda^{2k}\right) \right) \right)
\]
\[
=f'(1)+ \lim_{n \to \infty} \sum_{k=0}^n \frac{c}{\left(\lambda^2\right)^k}= f'(1)+c \sum_{n=0}^{\infty} \left(\lambda^{-2}\right)^{n} < \infty. 
\]
which contradicts the assumption that $f'$ is unbounded from above. The proof of the theorem is complete.
\end{proof}

The probably most important Bregman $f$-divergences on $\hP_n$ correspond to the following functions: $x\mapsto x^p$ $(p>1)$, $x\mapsto x\log x -x$, $x\mapsto -\log x$.
Unfortunately, the general theorem above covers only the case of the first type of functions. Indeed, the second function has derivative which is neither bounded from below nor unbounded from above and the derivative of the third one is not bounded from below. Fortunately, the Bregman divergence related to the third function, i.e., Stein's loss is of the form \eqref{E:M4} and the corresponding preservers were characterized in \cite{ML15b}. By \cite[Theorem 2]{ML15b} a surjective map $\phi:\hP_n \to \hP_n$ preserves the Stein's loss if and only if there is an invertible linear or conjugate linear operator $T:\mathbb C^n\to \mathbb C^n$ such that $\phi$ is of the form 
\[
\phi(A)=TAT^*,\quad A\in \hP_n.
\]

In what follows we characterize the preservers of Umegaki's relative entropy or, in other words, von Neumann divergence which is the Bregman divergence corresponding to the function $x\mapsto x\log x -x$. It is clear that this divergence equals
\[
\tr \left(A(\log A-\log B)-(A-B)\right), \quad A,B\in \hP_n.
\]
The result reads as follows.

\begin{theorem}
Let $\phi:\hP_n \to \hP_n$ be a surjective map which satisfies
\begin{equation}
\begin{gathered}
\tr \left(\phi(A)\left(\log \phi(A)-\log \phi(B)\right)-\left(\phi(A)-\phi(B)\right)\right)\\=
\tr (A(\log A-\log B)-(A-B)), \quad A,B\in \hP_n.
\end{gathered}
\end{equation}
Then there exists a unitary or antiunitary operator $U: \C^n \rightarrow \C^n$ such that $\phi$ is of the form
\[
 \phi(A)=UAU^{*}, \quad  A \in \hP_n.
\]
\end{theorem}

\begin{proof}
We begin with the following general observation. Assume that $f$ is a strictly convex differentiable function on $(0,\infty)$. We assert that for
any $B, C \in \hP_n$, the set
\[
 \left\{ H_f(A,B)-H_f(A,C) | A \in \hP_n \right\}
\]
is bounded from below if and only if $f'(B) \leq f'(C).$
Indeed, we have
\[
H_f(A,B)-H_f(A,C)
\]
\[
 =\tr \left(f'(C)-f'(B)\right)A + \tr \left(f(C)-f(B) -\left(f'(C)C-f'(B)B\right)\right)
\]
which is easily seen to be bounded from below if and only if $f'(C)-f'(B)\geq 0$. 

Therefore, for any surjective (and hence, by the strict convexity of $f$, bijective) map $\phi:\hP_n\to \hP_n$ which preserves the Bregman $f$-divergence, we obtain that $\phi$ has the property that
\[
f'(B)\leq f'(C) \Longleftrightarrow f'\left(\phi(B)\right)\leq f'\left(\phi(C)\right).
\] 
This means that the transformation $A\mapsto f'\left(\phi\left({f'}^{-1}(A)\right)\right)$ is an order automorphism of the set of all self-adjoint operators on $\mathbb C^n$ with spectrum contained in the range of $f'$.

In our particular case we have $f'=\log$, therefore $A\mapsto \log\left(\phi\left(e^A\right)\right)$ is an order automorphism of the set of all self-adjoint operators on $\mathbb C^n$. By \cite[Theorem 2]{m01} we have an invertible linear or conjugate-linear operator $T:\mathbb C^n\to \mathbb C^n$ and a self-adjoint linear operator $X:\mathbb C^n\to \mathbb C^n$ such that
\[
\log\left(\phi\left(e^A\right)\right)=TAT^*+X
\]
or 
\begin{equation}\label{E:M8}
\phi(A)=e^{T\log A T^*+X}, \quad A\in \hP_n.
\end{equation}
We assume that $T$ is linear, the conjugate-linear case is similar, not difficult to handle.
Consider the polar decomposition $T=UP$ of $T$ where $U$ is unitary and $P=\sqrt{T^*T}$ is positive definite.
As already mentioned, the unitary similarity transformation $A\mapsto UAU^*$ preserves Bregaman divergences.
Since
\[
\phi(A)=e^{UP(\log A) PU^*+X}=Ue^{P (\log A) P+U^*XU}U^*, \quad A\in \hP_n,
\]
it follows that in \eqref{E:M8} we can and do assume that $T$ is a positive definite operator. We prove that necessarily  $T=I$ holds. To see this, first assume that $T$ has an eigenvalue which is greater than 1. We know that
\begin{equation}\label{E:M3}
\begin{gathered}
\tr \left(e^{T(\log A) T+X}(T(\log A)T-T(\log B)T)-\left(e^{T(\log A) T+X}-e^{T(\log B) T+X}\right)\right)\\=
\tr (A(\log A-\log B)-(A-B)), \quad A,B\in \hP_n.
\end{gathered}
\end{equation}
Fixing $A,B$, write $tB$ in the place of $B$ where $t>0$ is arbitrary.
The above equality gives us that
\begin{equation}\label{E:M9}
a\log t+ \tr e^{(\log t)T^2 +T(\log B) T +X} +c=
d\log t+et +f, \quad t>0
\end{equation}
holds for some real constants $a,b,c,d,e,f$.
Select a real number $\mu$ such that $\mu I\leq T(\log B) T +X$.
Let $\l$ be an eigenvalue of $T$ which is greater than $1$ and let $x$ be a corresponding unit eigenvector. Denote by $P_x$ the rank-one projection onto the subspace generated by $x$.
Then $\l^2 P_x\leq T^2$, so for $t\geq 1$ we have
$(\log t)\l^2 P_x+\mu I \leq (\log t)T^2 +T(\log B) T +X$. By the monotonicity of trace functions (see \cite[2.10. Theorem]{carlen}) this implies that
\[
\tr e^{(\log t)\l^2 P_x+\mu I}
\leq \tr e^{(\log t)T^2 +T(\log B) T +X}, \quad t\geq 1.
\]
Therefore, the function $t\mapsto \tr e^{(\log t)T^2 +T(\log B) T +X}$, $t\geq 1$ can be minorized by a function $\alpha t^{\l^2} +\beta$ with some positive $\alpha$ and real number $\beta$. Since $\l^2>1$, considering the equality \eqref{E:M9} and letting $t$ tend to infinity, we easily obtain a contradiction. Therefore, the eigenvalues of $T$ are all less than or equal to 1. However, the inverse of $\phi$ also preserves the von Neumann divergence and we have
\[
\phi^{-1}(A)=e^{T^{-1}(\log A) T^{-1}-T^{-1}XT^{-1}}, \quad A\in \hP_n
\]
It follows that the eigenvalues of $T^{-1}$ are also not greater than 1. We conclude that $T=I$.

We finally prove that $X=0$. Let $A=I$ and $B$ be any element of $\hP_n$ which commutes with $X$. We obtain from \eqref{E:M3} that 
\[
\tr e^X(-\log B+B-I)=\tr (-\log B+B-I). 
\]
By the properties of the function $x\mapsto -\log x +x -1$, $x>0$ (strictly increasing for $x\geq 1$, and takes the value 0 at 1) it is easy to see that
any positive semidefinite operator $D$ which commutes with $X$ can be written as $D=-\log B +B-I$ with some positive definite $B$ which commutes with $X$. Consequently, we have
\[
\tr e^X D=\tr D
\]
for any positive semidefinite operator $D$ on $\mathbb C^n$.
This clearly implies that $e^X=I$, i.e., $X=0$. The proof of the theorem is complete.
\end{proof}

Let us mention here that the structure of all surjective maps on the set of nonsingular density operators (i.e., positive definite operators with trace 1) which preserve the relative entropy was determined in \cite[Theorem 3]{m11}. The conclusion there is analogous to the conclusion here, those maps are unitary or antiunitary congruence transformations. Furthermore, the fundamental idea of the proof there is basically the same as here although the details are rather different.

We now turn to the preservers of Jensen divergences. Our general result reads as follows.  

\begin{theorem} \label{jen}
Let $f$ be a differentiable strictly convex function on $(0,\infty)$, assume $\lim_{x\to 0+} f(x)$ exists and finite and $f'$ is unbounded from above. Pick $\lambda \in (0,1).$ If $\phi: \hP_n \rightarrow \hP_n$ is a surjective map which satisfies
\[
J_{f, \lambda}\left(\phi(A), \phi(B)\right)=J_{f, \lambda}\left(A, B\right), \quad A,B \in \hP_n,
\]
then there exists a unitary or antiunitary operator $U: \C^n \rightarrow \C^n$ such that $\phi$ is of the form
\[
 \phi(A)=UAU^{*}, \quad A \in \hP_n.
\]
\end{theorem}

\begin{proof}
Observe first that, by the assumptions on the function, $f$ is monotonically increasing for large enough values of its variable. Next, we verify the following.

\textbf{Claim B.} 
For any $B, C \in \hP_n$, the set
\[
 \left\{ J_{f, \lambda}(A,B)-J_{f, \lambda} (A,C) | A \in \hP_n \right\}
\]
is bounded from below if and only if $B \leq C.$

To prove the claim first observe that 
\[
J_{f, \lambda}(A,B)-J_{f, \lambda} (A,C)=(1-\lambda)\left( \tr f(B)-\tr f(C)\right)+
\]
\[
+\tr f \left( \lambda A +(1-\lambda) C \right)-\tr f \left( \lambda A +(1-\lambda) B \right), \quad A\in \hP_n. 
\]
Assume now that $B\leq C$ and that there is a sequence $(A_k)$ in $\hP_n$ such that
$$
\tr f \left( \lambda A_k +(1-\lambda) C \right)-\tr f \left( \lambda A_k +(1-\lambda) B \right)\to -\infty.
$$
Denote $\mu_i^{(k)}$, $i=1,\ldots ,n$ the eigenvalues of $\lambda A_k +(1-\lambda) B$ and $\l_i^{(k)}$, $i=1,\ldots ,n$ the eigenvalues of $\lambda A_k +(1-\lambda) C$ both ordered in increasing order. By the Weyl inequality (see e.g.  \cite[4.3.3. Corollary]{HJ}) it follows that $\mu_i^{(k)}\leq \l_i^{(k)}$ for all $i$ and $k$. Moreover, we have $\sum_i \left(f\left(\l_i^{(k)}\right)- f\left(\mu_i^{(k)}\right)\right)\to -\infty$ which implies that the sequence $\left(f\left(\l_i^{(k)}\right)- f\left(\mu_i^{(k)}\right)\right)$ is not bounded from below for some $i=1,\ldots, n$ and hence it has a subsequence $\left(f\left(\l_i^{(k_l)}\right)- f\left(\mu_i^{(k_l)}\right)\right) \to -\infty$. Since $f$ is bounded from below, we deduce that $f\left(\mu_i^{(k_l)}\right)\to \infty$. This implies that $\mu_i^{(k_l)} \to \infty$ and hence $f\left(\l_i^{(k_l)}\right)-f\left(\mu_i^{(k_l)}\right)\geq 0$ for all but finitely many indexes $l$ which is a contradiction.

Conversely, if $B \not\leq C,$ then there exists a unit vector $x \in \C^n$ such that $\left< Bx,x\right> > \left< Cx,x \right>.$ Set $\varepsilon=\left< (B-C)x,x\right>$ and let $P_x$ denote the orthogonal projection onto the one-dimensional subspace generated by $x.$ Let $m$ be a positive number such that $mI -(1-\lambda)C$ is positive definite. For any $t >0$ set
\[
A_t:=\frac{1}{\lambda}\left(m I+t P_x-(1-\lambda)C \right).
\]
Now
\begin{equation}\label{atir}
\begin{gathered}
J_{f, \lambda}(A_t,B)-J_{f, \lambda} (A_t,C)=(1-\lambda)\left( \tr f(B)-\tr f(C)\right)\\
+\tr f \left( \lambda A_t +(1-\lambda) C \right)-\tr f \left( \lambda A_t +(1-\lambda) B \right)
\\
=(1-\lambda)\left( \tr f(B)-\tr f(C)\right)\\+ \tr f \left( m I +t P_x\right)-\tr f \left( m I +t P_x+(1-\lambda)(B-C) \right).
\end{gathered}
\end{equation}
Let $\{x, y_2, y_3, \dots, y_n\}$ be an orthonormal basis in $\C^n.$ For $2 \leq j \leq n,$ denote by $a_{jj}$ the diagonal matrix elements of $ m I +t P_x+(1-\lambda)(B-C)$ relative to this basis, that is,
\[
 a_{jj}:=\left<\left(m I +t P_x+(1-\lambda)(B-C)\right) y_j, y_j\right>=\left<\left(m I+(1-\lambda)(B-C)\right) y_j, y_j\right>.
\]
Note that $a_{jj}$ is independent of $t$ for $2 \leq j$ and 
\[
\left<\left(m I +t P_x+(1-\lambda)(B-C)\right) x, x\right>=m+t+(1-\lambda)\varepsilon. 
\]
By Peierls inequality (see \cite[2.9. Theorem]{carlen}), for the convex function $f$ we have
\[
f(m+t+(1-\lambda)\varepsilon)+\sum_{j=2}^n f(a_{jj}) \leq \tr f \left( m I +t P_x+(1-\lambda)(B-C) \right).
\]
On the other hand, it is apparent that
\[
\tr f \left( m I +t P_x\right)= f(m+t)+ (n-1)f(m).
\]
Therefore, by (\ref{atir}),
\[
J_{f, \lambda}(A_t,B)-J_{f, \lambda} (A_t,C) \leq f(m+t)-f(m+t+(1-\lambda)\varepsilon)+K(B,C),
\]
where
\[
K(B,C)=(1-\lambda)\left( \tr f(B)-\tr f(C)\right) + (n-1)f(m)-\sum_{j=2}^n f(a_{jj})
\]
is independent of the parameter $t.$
Since
$f'$ is unbounded from above, hence 
\[
\lim_{t \to \infty }f(m+t)-f(m+t+(1-\lambda)\varepsilon)=- \infty, 
\]
and this completes the proof of our claim.

Using the characterization of the order given in Claim B, we see that the transformation $\phi$ is an order automorphism of $\hP_n$ and hence it is of the form
\[
 \phi(A)= T A T^{*}, \quad A \in \hP_n,
\]
where $T$ is an invertible linear or conjugate-linear operator on $\mathbb C^n$. We consider only the case where $T$ is linear and prove that then $T$ is necessarily unitary.

As already mentioned, the unitary-antiunitary congruence transformations preserve the Jensen divergences. Hence, by polar decomposition, we can assume that $T$ is a positive definite operator.

Any power of $\phi$ is also a divergence preserver, so for every positive integer $n$ we have that
\begin{equation}\label{E:M10}
\begin{gathered}
\tr \l f(T^n AT^n ) +(1-\l) f(T^n BT^n)-f(T^n(\l A +(1-\l)B)T^n)\\=
\tr \l f(A) +(1-\l) f(B)-f(\l A +(1-\l)B), \quad A,B\in \hP_n.
\end{gathered}
\end{equation}
Since $f$ is continuously extendible onto $[0,\infty)$, thus we can insert any positive semidefinite operators $A,B$ in the equality above. Assume $T$ has an eigenvalue, say $s$, which is greater than 1 and $x$ is a corresponding unit eigenvector. As before, denote by $P_x$ the orthogonal projection onto the subspace generated by $x$. Plug $A=P_x$ and $B=0$ into \eqref{E:M10}. We have
\[
\l f\left(s^{2n}\right)-f\left(\l s^{2n}\right)=c
\]
with some constant $c$.
One can easily check that the function $t\mapsto \l f(t)-f(\l t)$ is monotonically increasing (just differentiate and use that $f'$ is increasing). Since $s^{2n}\to \infty$, it follows that
$$ 
\l f(t)-f(\l t)=c,\quad t>0.
$$
We deduce $\l f'(t)=f'(\l t)\l$, $t>0$ and this implies that $f'$ is constant. We obtain that $f$ is affine, a contradiction. It follows that $T$ has no eigenvalue which is greater than 1. Since $\phi^{-1}$ also preserves the $\l-f$ Jensen divergence, we have that the eigenvalues of $T^{-1}$ are also not greater than 1. It follows that $T=I$ and the proof of the theorem is complete.
\end{proof}

Let us again consider the three most important examples 
$x\mapsto x^p$ $(p>1)$, $x\mapsto x\log x -x$, $x\mapsto -\log x$ of generating functions.
The first two do satisfy the conditions in our theorem, hence the corresponding preservers are unitary-antiunitary congruence transformations. As for the third one, it does not satisfy the conditions (not bounded below), but the Jensen divergence in that case is of the form \eqref{E:M4}, see the discussion before Proposition~\ref{P:2}. By \cite[Theorem 2]{ML15b}, a surjective map $\phi:\hP_n \to \hP_n$ preserves the corresponding Jensen divergence (i.e., Chebbi-Moakher log-determinant $\alpha$-divergence) if and only if there is an invertible linear or conjugate linear operator $T:\mathbb C^n\to \mathbb C^n$ such that $\phi$ is of the form 
\[
\phi(A)=TAT^*,\quad A\in \hP_n.
\]

In closing the paper we finally remark that the main result in our paper \cite{ML14f} where we have described the structure of all transformations on any dense subset of density operators that preserve the Holevo bound is closely related to Theorem~\ref{jen} in the particular case where the function $f$ is $x\mapsto x\log x -x$, $x>0$. Indeed, the Holevo bound of the ensemble $\{\l, 1-\l\}$ is just the corresponding Jensen divergence. We admit that the proof in \cite{ML14f} is of completely different character.

\bibliographystyle{amsplain}

\begin{thebibliography}{99}

\bibitem{carlen} E. Carlen, \emph{Trace inequalities and quantum entropy: an introductory course,} Contemp. Math. {\bf 529} (2010), 73-140.

\bibitem{ChM}
Z. Chebbi and M. Moakher, \emph{Means of hermitian positive-definite matrices based on the
log-determinant $\alpha$-divergence function,} Linear Algebra Appl. \textbf{436} (2012), 1872--1889.

\bibitem{ML14a}
O. Hatori and L. Moln\'ar,
\emph{Isometries of the unitary groups and Thompson isometries of the spaces of invertible positive elements in $C^*$-algebras,} J. Math. Anal. Appl. \textbf{409} (2014), 158--167.

\bibitem{HJ}
R.A. Horn and C.R. Johnson, \emph{Matrix Analysis,} Cambridge University Press, 1990.

\bibitem{Jarai}
A. J\'arai, \emph{Regularity Properties of Functional Equations in Several Variables,} Advances in Mathematics (Springer), Springer, New York, 2005. 

\bibitem{m01} L. Moln\'ar, \emph{Order-automorphisms of the set of bounded observables,} J. Math. Phys. {\bf 42} (2001), 5904-5909.

\bibitem{m11} L. Moln\'ar, \emph{Order automorphisms on positive definite operators and a few applications,} Linear Algebra Appl. {\bf434} (2011), 2158-2169.

\bibitem{ML15a}
L. Moln\'ar,
\emph{Jordan triple endomorphisms and isometries of spaces of positive definite matrices,}
Linear Multilinear Alg. \textbf{63} (2015), 12--33.

\bibitem{ML13j}
L. Moln\'ar,
\emph{General Mazur-Ulam type theorems and some applications,}
in Operator Semigroups Meet Complex Analysis, Harmonic Analysis and Mathematical Physics, W. Arendt, R. Chill, Y. Tomilov (Eds.), Operator Theory: Advances and Applications, Vol. 250, to appear.

\bibitem{ML14f}
L. Moln\'ar and G. Nagy,
\emph{Transformations on density operators that leave the Holevo bound invariant,}
Int. J. Theor. Phys. \textbf{53} (2014), 3273--3278.

\bibitem{mnsz13} L. Moln\'ar, G. Nagy and P. Szokol, \emph{Maps on density operators preserving quantum $f$-divergences,} Quantum Inf. Process. {\bf 12} (2013), 2309-2323.

\bibitem{ML15b}
L. Moln\'ar and P. Szokol,
\emph{Transformations on positive definite matrices preserving generalized distance measures,} Linear Algebra Appl. 
\textbf{466} (2015), 141--159.

\bibitem{NB}
F. Nielsen and R. Bhatia (Eds),
\emph{Matrix Information Geometry,} Springer, Heidelberg, 2013.

\bibitem{pv15} J. Pitrik and D. Virosztek, \emph{On the joint convexity of the Bregman divergence of matrices,} Lett. Math. Phys. {\bf 105} (2015), 675-692.

\bibitem{R}
R.T. Rockafellar, \emph{Convex Analysis,} Princeton University Press, Princeton, 1970.

\bibitem{Sra}
S. Sra, \emph{Positive definite matrices and the S-divergence,} preprint, arXiv math.FA-1110.1773v4. Shorter version of the manuscript is submitted to Proc. Amer. Math. Soc.

\end{thebibliography}

\end{document}